\documentclass[10pt]{article}
\usepackage{amssymb,amsmath,amsthm,verbatim}

\newtheorem{theorem}{Theorem}

\newtheorem{lemma}[theorem]{Lemma}

\newtheorem{conjecture}[theorem]{Conjecture}

\theoremstyle{definition}

\DeclareMathOperator{\Bin}{Bin}
\newcommand{\p}{\mathcal{P}}
\newcommand{\A}{\mathcal{A}}
\newcommand{\B}{\mathcal{B}}
\newcommand{\C}{\mathcal{C}}
\newcommand{\D}{\mathcal{D}}
\newcommand{\E}{\mathcal{E}}
\newcommand{\cI}{\mathcal{I}}

\begin{document}
\title{An isoperimetric inequality for antipodal subsets of the discrete cube}
\author{David Ellis and Imre Leader}
\date{14th September 2016}
\maketitle
\begin{abstract}
We say a family of subsets of $\{1,2,\ldots,n\}$ is {\em antipodal} if it is closed under taking complements. We prove a best-possible isoperimetric inequality for antipodal families of subsets of $\{1,2,\ldots,n\}$ (of any size). Our inequality implies that for any $k \in \mathbb{N}$, among all such families of size $2^k$, a family consisting of the union of two antipodal $(k-1)$-dimensional subcubes has the smallest possible edge boundary.
\end{abstract}

\section{Introduction}
Isoperimetric questions are classical objects of study in mathematics. In general, they ask for the minimum possible `boundary-size'
of a set of a given `size', where the exact meaning of these words
varies according to the problem.

The classical isoperimetric
problem in the plane asks for the minimum possible perimeter of a
shape in the plane with area 1. The answer, that it is best to take a
circle, was `known' to the ancient Greeks, but it was not until the
19th century that this was proved rigorously, by Weierstrass in a series of lectures in the
1870s in Berlin.

The isoperimetric problem has been solved for $n$-dimensional Euclidean space $\mathbb{E}^n$, for the $n$-dimensional unit sphere $\mathbb{S}^n : = \{x \in \mathbb{R}^{n+1}:\ \sum_{i=1}^{n+1}x_i^2 = 1\}$, and for $n$-dimensional hyperbolic space $\mathbb{H}^n$ (for all $n$), with the natural notion of boundary in each case, corresponding to surface area for sufficiently `nice' sets. (For background on isoperimetric problems, we refer the reader to the book of Burago and Zalgaller \cite{burago}, the surveys of Osserman \cite{osserman} and of Ros \cite{ros-survey}, and the references therein.) One of the most well-known open problems in the area is to solve the isoperimetric problem for $n$-dimensional real projective space $\mathbb{RP}^n$, or equivalently for antipodal subsets of the $n$-dimensional sphere $\mathbb{S}^n$. (We say a subset $\mathcal{A} \subseteq \mathbb{S}^n$ is {\em antipodal} if $\mathcal{A} = -\mathcal{A}$.) The conjecture can be stated as follows.

\begin{conjecture}
\label{conj:proj}
Let $n \in \mathbb{N}$ with $n \geq 2$, and let $\mu$ denote the $n$-dimensional Hausdorff measure on $\mathbb{S}^n$. Let $\A \subseteq \mathbb{S}^n$ be open and antipodal. Then there exists a set $\mathcal{B} \subseteq \mathbb{S}^n$ such that $\mu(\B) = \mu(\A)$, $\sigma(\B) \leq \sigma(\A)$, and
$$\mathcal{B} = \{x \in \mathbb{S}^n:\ \sum_{i=1}^{r} x_i^2 > a\}$$
for some $r \in [n]$ and some $a \in \mathbb{R}$.
\end{conjecture}

Here, if $\mathcal{A} \subseteq \mathbb{S}^n$ is an open set, then $\sigma(\mathcal{A})$ denotes the surface area of $\mathcal{A}$, i.e.\ the $(n-1)$-dimensional Hausdorff measure of the topological boundary of $\mathcal{A}$.

Only the cases $n=2$ and $n=3$ of Conjecture \ref{conj:proj} are known, the former being `folklore' and the latter being due to Ritor\'e and Ros \cite{ros}. In this paper, we prove a discrete analogue of Conjecture \ref{conj:proj}.

First for some definitions and notation. If $X$ is a set, we write $\mathcal{P}(X)$ for the power-set of $X$. For $n \in \mathbb{N}$, we write $[n]: = \{1,2,\ldots,n\}$, and we let $Q_n$ denote the graph of the $n$-dimensional discrete cube, i.e.\ the graph with vertex-set $\p([n])$, where $x$ and $y$ are joined by an edge if $|x \Delta y|=1$. If $\A \subseteq \p([n])$, we write $\partial \A$ for the {\em edge-boundary} of $\A$ in the discrete cube $Q_n$, i.e.\ $\partial \A$ is the set of edges of $Q_n$ which join a vertex in $\A$ to a vertex outside $\A$. We write $e(\A)$ for the number of edges of $Q_n$ which have both end-vertices in $\A$. We say that two families $\A,\B \subseteq \p([n])$ are {\em isomorphic} if there exists an automorphism $\sigma$ of $Q_n$ such that $\mathcal{B} = \sigma(\mathcal{A})$. Clearly, if $\mathcal{A}$ and $\mathcal{B}$ are isomorphic, then $|\partial \A| = |\partial \B|$.

The {\em binary ordering} on $\p([n])$ is defined by $x < y$ iff $\max(x \Delta y) \in y$. An {\em initial segment of the binary ordering on $\p([n])$} is the set of the first $k$ (smallest) elements of $\p([n])$ in the binary ordering, for some $k \leq 2^n$. For any $k \leq 2^n$, we write $\cI_{n,k}$ for the initial segment of the binary ordering on $\p([n])$ with size $k$.

Harper \cite{harper}, Lindsay \cite{lindsay}, Bernstein \cite{bernstein} and Hart \cite{hart} solved the edge isoperimetric problem for $Q_n$, showing that among all subsets of $\p([n])$ of given size, initial segments of the binary ordering on $\p([n])$ have the smallest possible edge-boundary. 

In this paper, we consider the edge isoperimetric problem for antipodal sets in $Q_n$. If $x \subseteq [n]$, we define $\overline{x}:= [n] \setminus x$, and if $\A \subseteq \p([n])$, we define $\overline{\A} := \{\overline{x}:\ x \in \A\}$. We say a family $\A \subseteq \p([n])$ is {\em antipodal} if $\A = \overline{\A}$. This notion is of course the natural analogue in the discrete cube of antipodality in $\mathbb{S}^n$; indeed, identifying $\p([n])$ with $\{-1,1\}^n \subseteq \sqrt{n} \cdot \mathbb{S}^{n-1} \subseteq \mathbb{R}^{n}$ in the natural way, $x \mapsto \overline{x}$ corresponds to the antipodal map $\mathbf{v} \mapsto -\mathbf{v}$.

We prove the following best-possible edge isoperimetric inequality for antipodal families.
\begin{theorem}
\label{thm:ant}
Let $\A \subseteq \p([n])$ be antipodal. Then
$$|\partial \A| \geq |\partial (\cI_{n,|\A|/2} \cup \overline{\cI_{n,|\A|/2}})|.$$
\end{theorem}

We remark that Theorem \ref{thm:ant} implies that if $\A \subseteq \p([n])$ is antipodal with $|\mathcal{A}| = 2^{k}$ for some $k \in [n-1]$, then $|\partial \A| \geq |\partial (\mathcal{S}_{k-1} \cup \overline{\mathcal{S}_{k-1}})|$, where $\mathcal{S}_{k-1} := \cI_{n,2^{k-1}} = \{x \subseteq [n]: x  \subseteq [k-1]\}$ is a $(k-1)$-dimensional subcube. In other words, a union of two antipodal subcubes has the smallest possible edge-boundary, over all antipodal sets of the same size.

To prove Theorem \ref{thm:ant}, it will be helpful for us to rephrase it slightly. Firstly, observe that for any $\A \subseteq \p([n])$, we have $\partial(\A^c) = \partial \A$, and that for any $k \leq 2^{n-1}$, the family $(\cI_{n,k} \cup \overline{\cI_{n,k}})^c$ is isomorphic to the family $\cI_{n,2^{n-1}-k} \cup \overline{\cI_{n,2^{n-1}-k}}$, via the isomorphism $x \mapsto x \Delta \{n\}$. Hence, by taking complements, it suffices to prove Theorem \ref{thm:ant} in the case $|\A| \leq 2^{n-1}$.

Secondly, for any family $\A \subseteq \p([n])$, we have
\begin{equation}\label{eq:internal-boundary} 2e(\A)+|\partial \A| = n|\A|,\end{equation}
so Theorem \ref{thm:ant} is equivalent to the statement that if $\A \subseteq \p([n])$ is antipodal, then
$$e(\A) \leq e(\cI_{n,|\A|/2} \cup \overline{\cI_{n,|\A|/2}}).$$
Note also that if $\mathcal{B}$ is an initial segment of the binary ordering on $\p([n])$ with $|\B| \leq 2^{n-2}$, then $\B \subseteq \{x \subseteq [n]:\ x \cap \{n-1,n\} = \emptyset\}$ and $\overline{\B} \subseteq \{x \subseteq [n]:\ \{n-1,n\} \subseteq x\}$, so $\B \cap \overline{\B} = \emptyset$ and $e(\B,\overline{\B}) = 0$. Moreover, it is easy to see that $\overline{\B}$ is isomorphic to $\B$, and therefore $e(\overline{\B}) = e(\B)$. Hence,
$$e(\B \cup \overline{\B}) = e(\B)+ e(\overline{\B}) = 2e(\B).$$
If $k,n \in \mathbb{N}$ with $k \leq 2^n$, we write $F(k) := e(\cI_{n,k})$. (It is easy to see that $F(k)$ is independent of $n$.) Putting all this together, we see that Theorem \ref{thm:ant} is equivalent to the following:
\begin{equation} \label{eq:aim} e(\A) \leq 2F(|\A|/2) \quad \forall \A \subseteq \p([n]):\ |\A| \leq 2^{n-1},\ \A \textrm{ is antipodal}.\end{equation} 

Now for a few words about our proof. In the special cases of $|\A| = 2^{n-1}$ and $|\A| = 2^{n-2}$, Theorem \ref{thm:ant} can be proved by an easy Fourier-analytic argument, but it is fairly obvious that this argument has no hope of proving the theorem for general set-sizes. Our proof of Theorem \ref{thm:ant} is purely combinatorial; we prove a stronger statement by induction on $n$. Our aim is to do induction on $n$ in the usual way: namely, by choosing some $i \in [n]$ and considering the upper and lower $i$-sections of $\mathcal{A}$, defined respectively by
$$\mathcal{A}_i^+ := \{x \in \p([n] \setminus \{i\}):\ x \cup \{i\} \in \mathcal{A}\},\quad \mathcal{A}_i^- := \{x \in \p([n] \setminus \{i\}):\ x \in \mathcal{A}\}.$$
However, a moment's thought shows that an $i$-section of an antipodal family need not be antipodal. For example, if $\mathcal{A} = \mathcal{S}_{k-1} \cup \overline{\mathcal{S}_{k-1}}$ (a union of two antipodal $(k-1)$-dimensional subcubes), then for any $i \geq k$, the $i$-section $\mathcal{A}_i^-$ consists of a single $(k-1)$-dimensional subcube, which is not an antipodal family. This rules out an inductive hypothesis involving antipodal families.

Hence, we seek a stronger statement, about arbitrary subsets of $\p([n])$; one which we can prove by induction on $n$, and which will imply Theorem \ref{thm:ant}. It turns out that the right statement is as follows. For any $\mathcal{A} \subseteq \p([n])$ (not necessarily antipodal), we define
$$f(\A) := 2e(\A) + |\A \cap \overline{\A}|.$$
To prove Theorem \ref{thm:ant}, it suffices to prove the following.
\begin{theorem}
\label{thm:suff}
For any $n \in \mathbb{N}$ and any $\A \subseteq \p([n])$ with $|\A| \leq 2^{n-1}$, we have
\begin{equation} \label{eq:suff} f(\A) \leq 2F(|\A|).\end{equation}
\end{theorem}
Indeed, assume that Theorem \ref{thm:suff} holds. Let $\mathcal{A} \subseteq \p([n])$ be antipodal with $|\A| \leq 2^{n-1}$. We have $\A_n^- = \overline{\A_n^+}$ and $|\A_n^+| = |\A|/2 \leq 2^{n-2}$, and so
\begin{align*} e(\A) & = e(\A_n^+)+e(\A_n^-) + |\A_n^+ \cap \A_n^-| = e(\A_n^+) + e(\overline{\A_n^+}) + |\A_n^+ \cap \overline{\A_n^+}|\\
& = 2e(\A_n^+) + |\A_n^+ \cap \overline{\A_n^+}| = f(\A_n^+) \leq 2F(|\A_n^+|) = 2F(|\A|/2),
\end{align*}
implying (\ref{eq:aim}) and so proving Theorem \ref{thm:ant}.

Note that the function $f$ takes the same value (namely, $k2^k$) when $\mathcal{A}$ is a $k$-dimensional subcube, as when $\mathcal{A}$ is the union of two antipodal $(k-1)$-dimensional subcubes. This is certainly needed in order for our inductive approach to work, by our above remark about the $i$-sections of the union of two antipodal subcubes.

We prove Theorem \ref{thm:suff} in the next section; in the rest of this section, we gather some additional facts we will use in our proof.

We will use the following lemma of Hart from \cite{hart}.
\begin{lemma}
\label{lemma:hart}
For any $x,y \in \mathbb{N} \cup \{0\}$, we have
$$F(x+y) - F(x) - F(y) \geq \min\{x,y\}.$$
Equality holds if $y$ is a power of 2 and $x \leq y$.
\end{lemma}

We will also use the following easy consequence of Lemma \ref{lemma:hart}.
\begin{lemma}
\label{lemma:hart-large}
Let $x,y \in \mathbb{N} \cup \{0\}$ and let $n \in \mathbb{N}$ such that $x+y \leq 2^{n}$, $y \geq 2^{n-1}$ and $y \leq 2^{n-1} + x$. Then
$$F(x+y) - F(y) -F(x) - y + 2^{n-1} \geq x.$$
\end{lemma}
\begin{proof}
Write $z: = y - 2^{n-1}$; then $z \leq x$ and $x+z \leq 2^{n-1}$. We therefore have
\begin{align*}
F(x+y) - F(y) -F(x) - y + 2^{n-1} & = F(2^{n-1} + x + z) - F(2^{n-1} + z) - F(x) - z\\
& = F(2^{n-1}) + F(x+z) + x + z\\
& \quad - F(2^{n-1}) - F(z) - z - F(x) - z\\
& = F(x+z)-F(x)-F(z) + x-z\\
& \geq \min\{x,z\} + x - z\\
& = x,
\end{align*}
where the last inequality uses Lemma \ref{lemma:hart}.
\end{proof}

We also need the following lemma, which says that for any family $\mathcal{A} \subseteq \p([n])$, one coordinate from every pair of coordinates is such that the upper and lower sections of $\mathcal{A}$ corresponding to that coordinate are `somewhat' close in size.

\begin{lemma}
\label{lemma:balanced}
Let $n \in \mathbb{N}$ with $n \geq 2$, and let $\A \subseteq \p([n])$. Then for any $1 \leq i < j \leq n$, we have
$$\min\{\left||\A_i^+| - |\A_i^-|\right|,\left||\A_j^+| - |\A_j^-|\right|\} \leq 2^{n-2}.$$
\end{lemma}
\begin{proof}
Without loss of generality, by considering $\{A \Delta S:\ A \in \A\}$ for some $S \subseteq \{i,j\}$, we may assume that $|\A_i^+| \leq |\A_i^-|$ and that $|\A_j^+| \leq |\A_j^-|$. Interchanging $i$ and $j$ if necessary, we may assume that
$|(\A_{i}^-)_j^+| \geq |(\A_i^+)_j^-|$. Then we have
\begin{align*} 0 \leq |\A_j^-| - |\A_j^+| & = |(\A_i^-)_j^-| + |(\A_i^+)_j^-| - |(\A_i^-)_j^+| - |(\A_i^+)_j^+|\\
& \leq |(\A_i^-)_j^-| - |(\A_i)^+)_j^+|\\
& \leq |(\A_i^-)_j^-|\\
& \leq 2^{n-2},\end{align*}
proving the lemma.
\end{proof}

We also need the following.

\begin{lemma}
\label{lemma:tech-ant}
Let $n \in \mathbb{N}$ and let $\C,\D \subseteq \p([n])$. Then
\begin{equation}\label{eq:ineq2} 2|\C \cap \D| + 2|\C \cap \overline{\D}| \leq |\C \cap \overline{\C}| + |\D \cap \overline{\D}| + 2\min\{|\C|,|\D|\}.\end{equation}
\end{lemma}
\begin{proof}
Note that both sides of the above inequality are invariant under interchanging $\C$ and $\D$, so it suffices to prove the lemma in the case $|\C| \leq |\D|$. By inclusion-exclusion, we have
$$2|\C \cap \D| + 2|\C \cap \overline{\D}| = 2|\C \cap (\D \cup \overline{\D})| + 2|\C \cap (\D \cap \overline{\D})| \leq 2|\C| + 2|\C \cap (\D \cap \overline{\D})|,$$
so it suffices to prove that
$$2|\C \cap (\D \cap \overline{\D})| \leq |\C \cap \overline{\C}| + |\D \cap \overline{\D}|.$$
Writing $\E = \D \cap \overline{\D}$, it suffices to prove that for any antipodal set $\E \subseteq \p([n])$, and any set $\C \subseteq \p([n])$, we have
$$2|\C \cap \E| \leq |\C \cap \overline{\C}| + |\E|.$$
This follows immediately from inclusion-exclusion again; indeed, we have
$$2|\C \cap \E| = |\C \cap \E| + |\C \cap \overline{\E}| = |\C \cap \E| + |\overline{\C} \cap \E| = |(\C \cap \overline{\C}) \cap \E| + |(\C \cup \overline{\C}) \cap \E| \leq |\C \cap \overline{\C}| + |\E|,$$
whenever $\E$ is antipodal.
\end{proof}

Finally, we note that for any $\A \subseteq \p([n])$, we have
\begin{align}
\label{eq:comp-f}
f(\A^c) & = 2e(\A^c) + |\A^c \cap \overline{\A^c}| = n|\A^c| - |\partial (\A^c)| +  |\A^c \cap \bar{\A}^c| \nonumber \\
&= n|\A| + n(|\A^c| - |\A|) - |\partial \A| + |(\A \cup \overline{\A})^c| \nonumber \\
& = n|\A| - |\partial \A| + n(2^n - 2|\A|)  + 2^n - |\A \cup \overline{\A}| \nonumber \\
& = 2e(\A) + n(2^n - 2|\A|) + 2^n - 2|\A| + |\A \cap \overline{\A}| \nonumber \\
& = f(\A)+  2(n+1) (2^{n-1} - |\A|).
\end{align}
Moreover, using (\ref{eq:internal-boundary}) and the fact that $(\cI_{n,k})^c$ is isomorphic to $\cI_{n,2^n-k}$, we have
\begin{align}
\label{eq:comp-rel}
2F(k) - 2F(2^n - k) & = kn - \partial(\cI_{n,k}) - \left((2^n-k)n - \partial((\cI_{n,k})^c)\right) \nonumber\\
& = kn - \partial(\cI_{n,k}) - \left((2^n-k)n - \partial(\cI_{n,k})\right) \nonumber \\
& = (2k-2^n)n
\end{align}
for any $k \leq 2^n$. It follows from (\ref{eq:comp-f}) and (\ref{eq:comp-rel}), by taking complements, that Theorem \ref{thm:suff} is equivalent to the inequality
\begin{equation} \label{eq:suff-large} f(\A) \leq 2F(|\A|) + 2|\A| - 2^n \quad \forall \A \subseteq \p([n]):\ |\A| \geq 2^{n-1}.\end{equation}

\section{Proof of Theorem \ref{thm:suff}}
Our proof is by induction on $n$. The base case $n=1$ of Theorem \ref{thm:suff} is easily checked. We turn to the induction step. Let $n \geq 2$, and assume that Theorem \ref{thm:suff} holds when $n$ is replaced by $n-1$. Let $\A \subseteq \p([n])$ with $|\A| \leq 2^{n-1}$. Observe that for any $i \in [n]$, we have
\begin{align}
\label{eq:gen-coord} f(\A) & = 2e(\A) + |\A \cap \overline{\A}|\nonumber\\
& = 2e(\A_i^+) + 2e(\A_i^-) + 2|\A_i^+ \cap \A_i^-| + |\A_i^+ \cap \overline{\A_i^-}| + |\A_i^- \cap \overline{\A_i^+}|\nonumber\\
& = 2e(\A_i^+) + 2e(\A_i^-) + 2|\A_i^+ \cap \A_i^-| + 2|\A_i^+ \cap \overline{\A_i^-}|.\end{align}

We now split into two cases.

{\em Case 1.} Firstly, suppose that there exists $i \in [n]$ such that $\max\{|\A_i^+|,|\A_i^-|\} \leq 2^{n-2}$. Without loss of generality, we may assume that this holds for $i=n$, i.e.\ that $\max\{|\A_n^+|,|\A_n^-|\} \leq 2^{n-2}$. We may also assume that $|\A_n^+| \leq |\A_n^-|$. Then, defining $\C : = \A_n^+ \subseteq \p([n-1])$ and $\D :=  \A_n^- \subseteq \p([n-1])$, and invoking (\ref{eq:gen-coord}) with $i=n$, we have
\begin{align}
\label{eq:master-1}
f(\A) & = 2e(\C) + 2e(\D) + 2|\C \cap \D| + 2|\C \cap \overline{\D}| \nonumber\\
& = f(\C) + f(\D) - |\C \cap \overline{\C}| - |\D \cap \overline{\D}| + 2|\C \cap \D| + 2|\C \cap \overline{\D}|.
\end{align}
We now apply the induction hypothesis to $\C$ and $\D$. Since $|\C| \leq |\D| \leq 2^{n-2}$, we may apply (\ref{eq:suff}), obtaining $f(\C) \leq 2F(|\C|)$ and $f(\D) \leq 2F(|\D|)$. Substituting the last two inequalities into (\ref{eq:master-1}), we obtain
\begin{align*}
f(\A) & \leq 2F(|\C|) + 2F(|\D|) + 2|\C \cap \D| + 2|\C \cap \overline{\D}| - |\C \cap \overline{\C}| - |\D \cap \overline{\D}|\\
& = 2F(|\C|+|\D|) - \big(2F(|\C|+|\D|) - 2F(|\C|) - 2F(|\D|)\big)\\
&\quad +2|\C \cap \D| + 2|\C \cap \overline{\D}| - |\C \cap \overline{\C}| - |\D \cap \overline{\D}|\\
& \leq 2F(|\A|) + 2|\C \cap \D| + 2|\C \cap \overline{\D}| - |\C \cap \overline{\C}| - |\D \cap \overline{\D}| - 2\min\{|\C|,|\D|\}\\
& = 2F(|\A|) + 2|\C \cap \D| + 2|\C \cap \overline{\D}| - |\C \cap \overline{\C}| - |\D \cap \overline{\D}| - 2|\C|\\
& \leq 2F(|\A|),
\end{align*}
where the second inequality follows from Lemma \ref{lemma:hart}, and the third inequality follows from Lemma \ref{lemma:tech-ant}. This completes the induction step in Case 1.

{\em Case 2.} Secondly, suppose that Case 1 does not occur, i.e.\ that $\max\{|\A_j^+|,|\A_j^-|\} > 2^{n-2}$ for all $j \in [n]$. By Lemma \ref{lemma:balanced}, there exists $i \in [n]$ such that $\left||\A_i^+| - |\A_i^-|\right| \leq 2^{n-2}$, and therefore 
$$2^{n-2} < \max\{|\A_i^+|,|\A_i^-|\} \leq \min\{|\A_i^+|,|\A_i^-|\} + 2^{n-2}.$$
Without loss of generality, we may assume that this holds for $i=n$, and that $|\A_n^+| \leq |\A_n^-|$, so that
$$2^{n-2} <  |\A_n^-| \leq 2^{n-2} + |\A_n^+|.$$
Defining $\C : = \A_n^+ \subseteq \p([n-1])$ and $\D :=  \A_n^- \subseteq \p([n-1])$ as before, and invoking (\ref{eq:gen-coord}) with $i=n$, we have
\begin{align}
\label{eq:master}
f(\A) & = 2e(\C) + 2e(\D) + 2|\C \cap \D| + 2|\C \cap \overline{\D}|\nonumber \\
& = f(\C) + f(\D) - |\C \cap \overline{\C}| - |\D \cap \overline{\D}| + 2|\C \cap \D| + 2|\C \cap \overline{\D}|.
\end{align}
Now, since $|\C| \leq |\D|$, we have $2|\C| \leq |\C|+|\D| = |\A| \leq 2^{n-1}$ and therefore $|\C| \leq 2^{n-2}$. On the other hand, we have $|\D| > 2^{n-2}$. Applying the induction hypothesis to $\C$ and $\D$ (using (\ref{eq:suff}) for $\C$ and (\ref{eq:suff-large}) for $\D$), we obtain $f(\C) \leq 2F(|\C|)$ and $f(\D) \leq 2F(|\D|) + 2|\D| - 2^{n-1}$; substituting these two inequalities into (\ref{eq:master}) yields
\begin{align*}
f(\A) & \leq 2F(|\C|) + 2F(|\D|) + 2|\D| - 2^{n-1} + 2|\C \cap \D| + 2|\C \cap \overline{\D}| - |\C \cap \overline{\C}| - |\D \cap \overline{\D}|\\
& = 2F(|\C|+ |\D|) - \big(2F(|\C|+|\D|) - 2F(|\C|) - 2F(|\D|) - 2|\D| + 2 \cdot 2^{n-2}\big)\\
&\quad +2|\C \cap \D| + 2|\C \cap \overline{\D}| - |\C \cap \overline{\C}| - |\D \cap \overline{\D}|\\
& \leq 2F(|\A|) + 2|\C \cap \D| + 2|\C \cap \overline{\D}| - |\C \cap \overline{\C}| - |\D \cap \overline{\D}| - 2|\C|\\
& \leq 2F(|\A|),
\end{align*}
where the second inequality uses Lemma \ref{lemma:hart-large}, applied with $x = |\C|$ and $y = |\D|$, and with $n-1$ in place of $n$, and the third inequality uses Lemma \ref{lemma:tech-ant}. This completes the induction step in Case 2, proving the theorem.

\section{Conclusion}
We feel that our proof of Theorem \ref{thm:suff} (and therefore of Theorem \ref{thm:ant}) is somewhat delicate, as it relies on the fact that, in the inductive step, the terms involving $F$ can be dealt with using the fortunate properties of the function $F$ (in Lemmas \ref{lemma:hart} and \ref{lemma:hart-large}), and the other terms can be dealt with using the elementary inequality in Lemma \ref{lemma:tech-ant}. We also note that there is a nested sequence of families (with one family of every possible size), each of which is extremal for Theorem \ref{thm:ant}. In contrast, the (conjectural) extremal families in Conjecture \ref{conj:proj} do not have this `nested' property. Hence, perhaps unfortunately, we feel that Theorem \ref{thm:ant}, and our proof thereof, may shed only a limited amount of light on Conjecture \ref{conj:proj}.

\subsubsection*{Acknowledgements}
We would like to thank an anonymous referee for suggesting the above proof of Lemma \ref{lemma:balanced}, which is simpler than our original argument.

\end{document}